\documentclass[11pt]{article}

\date{}

\usepackage{graphicx}
\usepackage{tabularx}
\usepackage{url}
\usepackage{amsthm}
\usepackage{amsmath}
\usepackage{amsfonts}
\usepackage{hyperref}
\usepackage{subcaption}

\usepackage{color}
\usepackage{xspace}
\usepackage{tikz}
\usetikzlibrary{decorations,decorations.pathmorphing,decorations.pathreplacing,fit,positioning,arrows}

\theoremstyle{plain}
\newtheorem{theorem}{Theorem}
\newtheorem{lemma}{Lemma}
\newtheorem{proposition}{Proposition}

\theoremstyle{definition}
\newtheorem{definition}{Definition}
\newtheorem{conjecture}{Conjecture}

\newtheorem{problem}{Open problem}

\theoremstyle{remark}

\usepackage{graphicx}
\usepackage{tabularx}
\usepackage{url}
\usepackage{amsmath}
\usepackage{amsfonts}
\usepackage{amssymb}
\usepackage{tikz}
\usepackage{hhline}
\usepackage{array}
\usetikzlibrary{decorations.pathreplacing}
\usetikzlibrary{shapes}
\usetikzlibrary{shapes.multipart}
\tikzstyle{vertex}=[circle,fill=black!100,text=white,inner sep=0.8mm]
\tikzstyle{point}=[circle,fill=black,inner sep=0.1mm]

\DeclareMathOperator{\Free}{Free}
\DeclareMathOperator{\Sun}{Sun}

\begin{document}
\title{Critical properties of bipartite permutation graphs}
\author{Bogdan Alecu\thanks{Mathematics Institute, University of Warwick, Coventry, CV4 7AL, UK. Email: B.Alecu@warwick.ac.uk} \and
	Vadim Lozin\thanks{Mathematics Institute, University of Warwick, Coventry, CV4 7AL, UK. Email: V.Lozin@warwick.ac.uk} \and
	Dmitriy Malyshev\thanks{Laboratory of Algorithms and Technologies for Networks Analysis, National Research University
Higher School of Economics, 136 Rodionova Str., 603093, Nizhny Novgorod, Russia. Email: dsmalyshev@rambler.ru}$~^,$\thanks{The work of Malyshev D.S. 
was conducted within the framework of the Basic Research Program at the National Research University Higher School of Economics (HSE).}}
\maketitle

\begin{abstract}
The class of bipartite permutation graphs enjoys many nice and important properties.
In particular, this class is critically important in the study of clique- and rank-width of graphs, because it is one of the minimal hereditary classes
of graphs of unbounded clique- and rank-width. It also contains a number of important subclasses, which are critical with respect to
other parameters, such as graph lettericity or shrub-depth, and with respect to other notions, such as well-quasi-ordering or complexity of algorithmic problems.
In the present paper we identify critical subclasses of bipartite permutation graphs of various types.
\end{abstract}


\section{Introduction}
A {\it graph property}, also known as a {\it class of graphs}, is an infinite set of graphs closed under isomorphism.
A property is {\it hereditary} if it is closed under taking induced subgraphs. The universe of hereditary properties is
rich and diverse, and it contains various classes of theoretical or practical importance, such as perfect graphs, interval graphs,
permutation graphs, bipartite graphs, planar graphs, threshold graphs, split graphs, graphs of bounded vertex degree, graphs of bounded clique-width, etc. 
It also contains all classes closed under taking subgraphs, minors, induced minors, vertex-minors, etc.

A class of particular interest in this paper is the class of \emph{bipartite permutation graphs}, i.e., graphs that are simultaneously
bipartite and permutation graphs. This class was introduced by Spinrad, Brandst\"adt, and Stewart \cite{BPG} in 1987
and since then it appears frequently in the mathematical and computer science literature. Partly, this is because bipartite permutation
graphs have a nice structure allowing solutions to many problems that are notoriously difficult for general graphs. For instance:
\begin{itemize}
\item the reconstruction conjecture, which is wide open in general, holds true for bipartite permutation graphs \cite{reconstruction};
\item many algorithmic problems that are generally NP-hard admit polynomial-time solutions when restricted to bipartite permutation graphs \cite{buffer,weed};
\item bipartite permutation graphs have bounded contiguity, a complexity measure, which is important in biological applications \cite{contiguity};
\item bipartite permutation graphs have a universal element of small order: there is a bipartite permutation graph with $n^2$ vertices containing all $n$-vertex bipartite permutation graphs as induced subgraphs.
\end{itemize}
On the other hand, in spite of the many attractive properties of bipartite permutation graphs, they represent a complex world. Indeed, in this world some algorithmic problems remain NP-hard,
for instance {\sc induced subgraph isomorphism}, and some parameters that measure complexity of the graphs take arbitrarily large values, for instance clique-width.
Moreover, the class of bipartite permutation graphs is critical with respect to clique-width in the sense that in every proper hereditary subclass of bipartite permutation graphs,
clique-width is bounded by a constant \cite{Loz11}. In other words, the class of bipartite permutation graphs is a minimal hereditary class of graphs of unbounded clique-width.
The same is true with respect to the notion of rank-width, because rank-width is bounded in a class of graphs if and only of clique-width is. Moreover, in the terminology of 
vertex-minors, bipartite permutation graphs constitute the {\it only} obstacle to bounding rank-width of bipartite graphs, because every bipartite graph of large rank-width 
contains a large universal bipartite permutation graph as a vertex-minor, see Corollary 3.9 in \cite{localcomp}.

This class, however, is not critical with respect to complexity of the  {\sc induced subgraph isomorphism} problem, because
the problem remains NP-hard when further restricting to the class of linear forests, a proper subclass of bipartite permutation graphs. One of the results of the present paper
is that the class of linear forests is a minimal hereditary class where the problem is NP-hard.

Is it always possible to find minimal ``difficult'' classes? In the universe of minor-closed classes of graphs the answer to this question is `yes', because graphs are well-quasi-ordered
under the minor relation \cite{minor-wqo}. In particular, in the family of minor-closed classes of graphs the planar graphs constitute a unique minimal class of unbounded tree-width \cite{planar}.
However, the induced subgraph relation is not a well-quasi-order, because it contains infinite antichains of graphs. As a result,
the universe of hereditary classes contains infinite strictly descending chains of classes. The intersection of all classes in such a chain is called a {\it limit class} and
a minimal limit class is called a {\it boundary class}.
Unfortunately, limit classes can be found even within bipartite permutation graphs. On the other hand, fortunately, there is only one boundary class in this universe, as we show
in the present paper. We also show that this unique boundary class is the only obstacle to finding minimal classes in the universe of bipartite permutation graphs.

In this paper, we study diverse problems and identify a variety of minimal ``difficult'' classes with respect to these problems.
In particular, in Section~\ref{sec:wqo} we identify the unique boundary subclass of bipartite permutation graphs.
Section~\ref{sec:parameters} is devoted to graph parameters and minimal classes where these parameters are unbounded.
In Section~\ref{sec:algo}, we  deal with algorithmic problems and prove that the linear forests constitute a minimal ``difficult'' classes for {\sc induced subgraph isomorphism}.
In Section~\ref{sec:uni}, we identify a minimal subclass of bipartite permutation graphs that do not admit a universal graph of linear order.
Various subclasses of bipartite permutation graphs that play an important role in this paper are presented in Section~\ref{sec:bpg}. 

All graphs in this paper are \emph{simple}, i.e., finite, undirected, without loops and without multiple edges. The vertex set and the edge set of a graph $G$ are denoted by $V(G)$ and $E(G)$, respectively.

As usual, $P_n,C_n,K_n$ denote the chordless path, the chordless cycle, and the complete graph with $n$ vertices, respectively.
Also, $K_{n,m}$ is the complete bipartite graphs with parts of size $n$ and $m$. By $S_{i,j,k}$, we denote the graph represented in Figure~\ref{fig:S}.

\begin{figure}[ht]
\begin{center} \begin{picture}(240,80)
\put(110,15){\circle*{3}}
\put(110,26){\circle*{3}}
\put(110,37){\circle*{3}}
\put(110,55){\circle*{3}}
\put(110,66){\circle*{3}}
\put(110,42){\circle*{1}}
\put(110,46){\circle*{1}}
\put(110,50){\circle*{1}}
\put(110,15){\line(0,1){11}}
\put(110,26){\line(0,1){11}}
\put(110,55){\line(0,1){11}}
\put(100,10){\circle*{3}}
\put(90,5){\circle*{3}}
\put(70,-5){\circle*{3}}
\put(60,-10){\circle*{3}}
\put(85,2){\circle*{1}}
\put(80,0){\circle*{1}}
\put(75,-2){\circle*{1}}
\put(110,15){\line(-2,-1){10}}
\put(100,10){\line(-2,-1){10}}
\put(70,-5){\line(-2,-1){10}}
\put(120,10){\circle*{3}}
\put(130,5){\circle*{3}}
\put(150,-5){\circle*{3}}
\put(160,-10){\circle*{3}}
\put(135,2){\circle*{1}}
\put(140,0){\circle*{1}}
\put(145,-2){\circle*{1}}
\put(110,15){\line(2,-1){10}}
\put(120,10){\line(2,-1){10}}
\put(150,-5){\line(2,-1){10}}
\put(113,26){$_1$}
\put(113,37){$_2$}
\put(113,55){$_{i-1}$}
\put(113,66){$_i$}
\put(98,3){$_1$}
\put(88,-2){$_2$}
\put(68,-12){$_{j-1}$}
\put(58,-17){$_j$}
\put(120,15){$_1$}
\put(130,10){$_2$}
\put(148,1){$_{k-1}$}
\put(162,-7){$_k$}
\end{picture}
\end{center}
\caption{The graph $S_{i,j,k}$ }
\label{fig:S}
\end{figure}
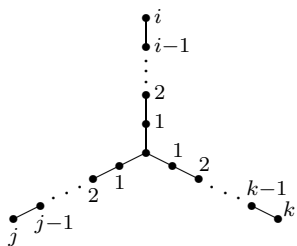

The disjoint union of two graphs $G_1$ and $G_2$ 
will be denoted by $G_1+G_2$,
and the disjoint union of $k$ copies of $G$ will be denoted by $kG$.

An {\it induced subgraph} of a graph $G$ is a subgraph obtained from $G$ by a (possibly empty) sequence of vertex deletions.
We say that $G$ contains a graph $H$ as an induced subgraph if $H$ is isomorphic to an induced subgraph of $G$.
Otherwise, we say that  $G$ is \emph{$H$-free} 
and that $H$ is a forbidden induced subgraph for $G$. 

It is well-known (and not difficult to see) that a class of graphs is hereditary if and only 
if it can be characterized by means of minimal forbidden induced subgraphs, i.e., graphs that do not belong to the class and which are minimal with this property (with respect to the induced subgraph relation).
The class containing no induced subgraphs from a set $M$ will be denoted $\Free(M)$, and given graphs $G_1, G_2, \dots$, we will write $\Free(G_1, G_2, \dots)$ to mean $\Free(\{G_1, G_2, \dots\})$.

\section{Bipartite permutation graphs and their subclasses}
\label{sec:bpg}
Throughout the paper we denote the class of all bipartite permutation graphs by $\cal BP$. 
This class was introduced in \cite{BPG} and also appeared in the literature under various other names such as 
proper interval bigraphs \cite{Pavol}, monotone graphs \cite{Haiko}, Parikh word representable graphs \cite{Parikh}.
This class admits various characterizations, many of which can be found in \cite{BPG}. 
For the purpose of the present paper, we are interested in the induced subgraph characterization and the structure of 
a `typical' graph in this class.

In terms of forbidden induced subgraphs the class of bipartite permutation graphs is precisely the class of
$$
(S_{2,2,2},\Sun_3,\Phi,C_3,C_5,C_6,C_7,\ldots)\mbox{-free graphs},
$$
where $\Sun_3$ and $\Phi$ are the graphs represented in Figure~\ref{fig:Sun}.
\begin{figure}[ht]
\begin{center}
\begin{picture}(130,50)
\setlength{\unitlength}{0.5mm}
\put(30,17){\circle*{3}}
\put(40,17){\circle*{3}}
\put(50,-3){\circle*{3}}
\put(50,7){\circle*{3}}
\put(50,27){\circle*{3}}
\put(60,17){\circle*{3}}
\put(70,17){\circle*{3}}
\put(40,17){\line(-1,0){10}}
\put(40,17){\line(1,1){10}}
\put(40,17){\line(1,-1){10}}
\put(60,17){\line(1,0){10}}
\put(60,17){\line(-1,1){10}}
\put(60,17){\line(-1,-1){10}}
\put(50,-3){\line(0,1){10}}
\end{picture}
\begin{picture}(130,50)
\setlength{\unitlength}{0.5mm}
\put(60,30){\circle*{3}}
\put(60,10){\circle*{3}}
\put(80,30){\circle*{3}}
\put(80,10){\circle*{3}}
\put(100,30){\circle*{3}}
\put(100,10){\circle*{3}}
\put(80,-10){\circle*{3}}

\put(60,30){\line(1,0){20}}
\put(100,30){\line(-1,0){20}}
\put(60,30){\line(0,-1){20}}
\put(100,10){\line(0,1){20}}
\put(80,10){\line(0,-1){20}}
\put(80,10){\line(0,1){20}}
\put(100,10){\line(-1,0){20}}
\put(60,10){\line(1,0){20}}
\end{picture}
\end{center}
\caption{Graphs $\Sun_3$ (left) and $\Phi$ (right)}
\label{fig:Sun}
\end{figure}
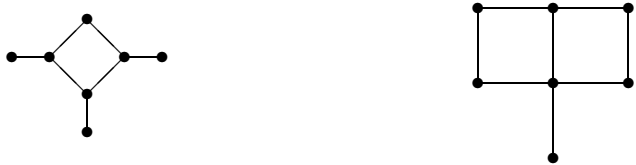

A `typical' bipartite permutation graph is represented in Figure~\ref{fig:H55}. 
We emphasize that this figure contains two representations of the {\it same} graph. 
In most of our considerations the square representation is more preferable and we 
denote a graph of this form with $n$ columns and $n$ rows by $H_{n,n}$.
The graph $H_{n,n}$ is typical in the sense that it contains all $n$-vertex bipartite permutation graphs as induced subgraphs, i.e., this is an $n$-universal 
bipartite permutation graph, which was shown in \cite{universal}.

\begin{figure}[ht]
\begin{center}
\begin{picture}(200,100)
\setlength{\unitlength}{0.2mm}
\put(0,0){\circle*{5}}

\put(50,0){\circle*{5}}
\put(100,0){\circle*{5}}
\put(150,0){\circle*{5}}
\put(200,0){\circle*{5}}
\put(250,0){\circle*{5}}

\put(100,50){\circle*{5}}
\put(150,50){\circle*{5}}
\put(200,50){\circle*{5}}
\put(250,50){\circle*{5}}
\put(50,50){\circle*{5}}
\put(0,50){\circle*{5}}

\put(300,50){\circle*{5}}
\put(300,0){\circle*{5}}

\put(350,50){\circle*{5}}
\put(350,0){\circle*{5}}

\put(200,0){\line(0,1){50}}
\put(200,0){\line(-1,1){50}}
\put(200,0){\line(-2,1){100}}
\put(200,0){\line(-3,1){150}}
\put(200,0){\line(-4,1){200}}

\put(150,0){\line(0,1){50}}
\put(150,0){\line(-1,1){50}}
\put(150,0){\line(-2,1){100}}
\put(150,0){\line(-3,1){150}}

\put(100,0){\line(0,1){50}}
\put(100,0){\line(-1,1){50}}
\put(100,0){\line(-2,1){100}}

\put(50,0){\line(0,1){50}}
\put(50,0){\line(-1,1){50}}
\put(0,0){\line(0,1){50}}

\put(250,0){\line(0,1){50}}
\put(250,0){\line(-1,1){50}}
\put(250,0){\line(-2,1){100}}
\put(250,0){\line(-3,1){150}}
\put(250,0){\line(-4,1){200}}

\put(300,0){\line(0,1){50}}
\put(300,0){\line(-1,1){50}}
\put(300,0){\line(-2,1){100}}
\put(300,0){\line(-3,1){150}}
\put(300,0){\line(-4,1){200}}
\put(350,0){\line(0,1){50}}
\put(350,0){\line(-1,1){50}}
\put(350,0){\line(-2,1){100}}
\put(350,0){\line(-3,1){150}}
\put(350,0){\line(-4,1){200}}
\end{picture}
\begin{picture}(150,100)
\setlength{\unitlength}{0.2mm}
\put(50,0){\circle*{5}}
\put(100,0){\circle*{5}}
\put(150,0){\circle*{5}}
\put(200,0){\circle*{5}}

\put(100,50){\circle*{5}}
\put(150,50){\circle*{5}}
\put(200,50){\circle*{5}}
\put(50,50){\circle*{5}}

\put(50,100){\circle*{5}}
\put(100,100){\circle*{5}}
\put(150,100){\circle*{5}}
\put(200,100){\circle*{5}}
\put(50,150){\circle*{5}}
\put(100,150){\circle*{5}}
\put(150,150){\circle*{5}}
\put(200,150){\circle*{5}}
\put(200,0){\line(0,1){50}}
\put(200,0){\line(-1,1){50}}
\put(200,0){\line(-2,1){100}}
\put(200,0){\line(-3,1){150}}

\put(150,0){\line(0,1){50}}
\put(150,0){\line(-1,1){50}}
\put(150,0){\line(-2,1){100}}

\put(100,0){\line(0,1){50}}
\put(100,0){\line(-1,1){50}}

\put(50,0){\line(0,1){50}}

\put(200,50){\line(0,1){50}}
\put(200,50){\line(-1,1){50}}
\put(200,50){\line(-2,1){100}}
\put(200,50){\line(-3,1){150}}

\put(150,50){\line(0,1){50}}
\put(150,50){\line(-1,1){50}}
\put(150,50){\line(-2,1){100}}

\put(100,50){\line(0,1){50}}
\put(100,50){\line(-1,1){50}}

\put(50,50){\line(0,1){50}}

\put(200,100){\line(0,1){50}}
\put(200,100){\line(-1,1){50}}
\put(200,100){\line(-2,1){100}}
\put(200,100){\line(-3,1){150}}

\put(150,100){\line(0,1){50}}
\put(150,100){\line(-1,1){50}}
\put(150,100){\line(-2,1){100}}

\put(100,100){\line(0,1){50}}
\put(100,100){\line(-1,1){50}}

\put(50,100){\line(0,1){50}}

\end{picture}
\end{center}
\caption{Universal bipartite permutation graph $H_{n,n}$ for $n=4$.}
\label{fig:H55}
\end{figure}
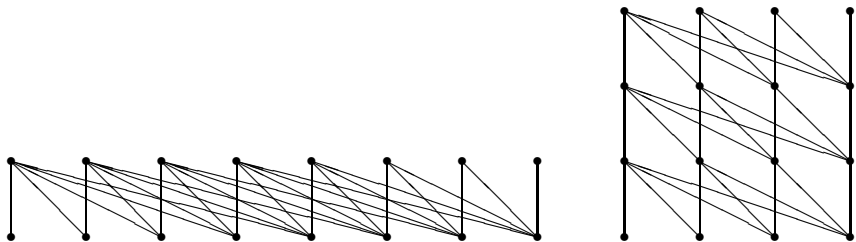

Below, we list a number of subclasses of bipartite permutation graphs that play an important role in this paper.
\begin{itemize}
\item {\it Chain graphs} are bipartite graphs for which the vertices in each part are linearly ordered under inclusion of their neighbourhoods.
These are precisely $2K_2$-free bipartite graphs.
\item {\it Graphs of vertex degree at most $1$} are graphs in which every connected component is either $K_2$ or $K_1$. Alternatively, they can be
described as $(P_3,K_3)$-free graphs.
\end{itemize}
As we shall see in Section~\ref{sec:parameters}, both these classes are critical with respect to a parameter known as neighbourhood diversity.
Graphs of degree at most $1$ admit two important extensions.
\begin{itemize}
\item {\it Linear forests}, also known as {\it path forests}, are graphs in which every connected component is a path $P_k$, for some $k$. These are
precisely $(K_{1,3},C_3,C_4,C_5,\ldots)$-free graphs, or alternatively, $(K_{1,3},C_4)$-free bipartite permutation graphs.
\item {\it Star forests} are graphs in which every connected component is a star $K_{1,p}$, for some $p$. In the terminology of forbidden induced subgraphs
this class can be described as the class of  $(P_4,C_4)$-free bipartite permutation graphs.
\end{itemize}
Both linear forests and star forests are special cases of caterpillar forests.
\begin{itemize}
\item {\it Caterpillar forests} are graphs in which every connected component is a caterpillar, i.e., a tree containing a dominating path.
In terms of forbidden induced subgraphs caterpillar forests can be described as the class of  $C_4$-free bipartite permutation graphs.
\end{itemize}
An interesting class between star forests and bipartite permutation graphs is the class of $P_5$-free graphs: these are graphs in which every connected
component is a chain graph. The inclusion relationships between the above listed classes is represented in Figure~\ref{fig:inclusion}.
\begin{figure}
\begin{center}
\begin{picture}(400,150)
\put(200,130)
{\oval(300,20)
\makebox(0,0)
{Bipartite permutation graphs}}
\put(110,90){\oval(100,20)
\makebox(0,0)
{$P_5$-free bipartite}}
\put(275,90){\oval(150,20)
\makebox(0,0)
{Caterpillar forests}}
\put(110,10){\oval(100,20)
\makebox(0,0)
{Chain graphs}}
\put(185,50){\oval(100,20)
\makebox(0,0)
{Star forests}}
\put(300,50){\oval(100,20)
\makebox(0,0)
{Linear forests}}
\put(250,10){\oval(104,20)
\makebox(0,0)
{Graphs of degree $\le 1$}}
\put(110,120){\line(0,-1){20}}
\put(110,80){\line(0,-1){60}}
\put(275,120){\line(0,-1){20}}
\put(150,80){\line(0,-1){20}}
\put(220,80){\line(0,-1){20}}
\put(275,80){\line(0,-1){20}}
\put(220,40){\line(0,-1){20}}
\put(275,40){\line(0,-1){20}}
\end{picture}
\end{center}
\caption{Inclusion relationships between subclasses of bipartite permutation graphs}
\label{fig:inclusion}
\end{figure}
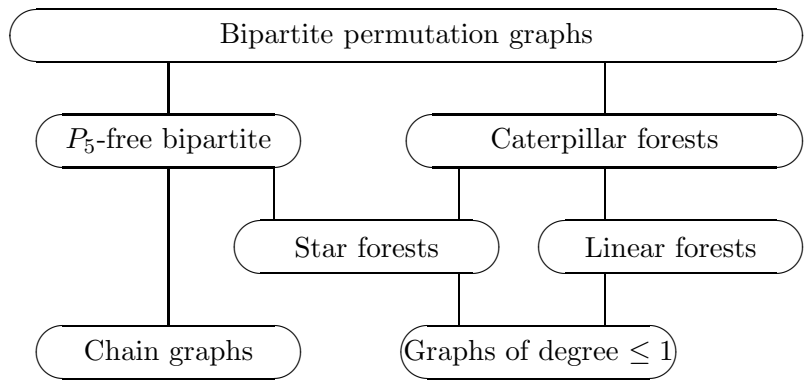

\section{Well-quasi-ordering and boundary classes}
\label{sec:wqo}

In this section, we look for critical classes of bipartite permutation graphs with respect to the notion of well-quasi-ordering by induced subgraphs.
We start with basic definitions.

A binary relation $\le$ on a set $W$ is a {\it quasi-order} (also known as {\it preorder})
if it is reflexive and transitive. Two elements $x,y \in W$ are said to be {\it comparable}
with respect to $\le$ if either $x\le y$ or $y\le x$. Otherwise, $x$ and $y$ are {\it incomparable}.
A set of pairwise comparable elements is called a {\it chain}, and a set of pairwise incomparable elements
is an {\it antichain}. If $x\le y$ and $y\not\le x$, we write $x<y$. A chain $x_1> x_2> \ldots$ is called {\it strictly descending}.
A quasi-order $(W,\le)$ is a {\it well-quasi-order} (``wqo'', for short) if it contains neither infinite strictly descending chains nor infinite antichains.

The celebrated result of Robertson and Seymour \cite{minor-wqo} states that the set of all simple graphs is well-quasi-ordered with respect to the minor relation.
However, the induced subgraph relation is not a wqo, as the cycles create an infinite antichain with respect to this relation.
This example does not apply to the class of bipartite permutation graphs, since all chordless cycles are forbidden in this class, except for $C_4$.
Nonetheless, graphs in this class are not well-quasi-ordered under induced subgraphs, because they contain the infinite antichain of \emph{$H$-graphs},
i.e., graphs of the form $H_k$ represented in Figure~\ref{fig:H}.

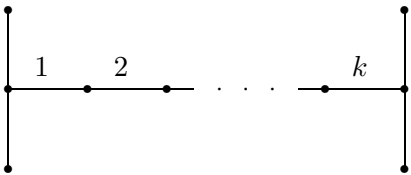
\begin{figure}[ht]
\begin{center}
\begin{picture}(250,85)
\put(50,40){\circle*{3}} \put(80,40){\circle*{3}} \put(110,40){\circle*{3}}
\put(130,40){\circle*{1}} \put(140,40){\circle*{1}} \put(150,40){\circle*{1}}
\put(170,40){\circle*{3}} \put(200,40){\circle*{3}} \put(50,70){\circle*{3}}
\put(50,10){\circle*{3}} \put(200,70){\circle*{3}} \put(200,10){\circle*{3}}
\put(50,40){\line(1,0){30}} \put(80,40){\line(1,0){30}} \put(110,40){\line(1,0){10}}
\put(160,40){\line(1,0){10}} \put(170,40){\line(1,0){30}}
\put(50,40){\line(0,1){30}} \put(50,40){\line(0,-1){30}}
\put(200,40){\line(0,1){30}} \put(200,40){\line(0,-1){30}} \put(60,45){1}
\put(90,45){2} \put(180,45){$k$}
\end{picture}
\end{center}
\caption{The graph $H_k$}
\label{fig:H}
\end{figure}

On the other hand, if we restrict ourselves to the class of chain graphs, we find ourselves in the well-quasi-ordered world,
because chain graphs have lettericity at most $2$, as we shall see in Section~\ref{sec:lettericity},
and graphs of bounded lettericity are known to be well-quasi-ordered by induced subgraphs \cite{Pet02}. Unfortunately, the boundary
separating wqo classes from non-wqo ones cannot be described in the terminology of minimal non-wqo
classes, because such classes do not exist. Indeed, if $\cal X$ is a non-wqo class, then it contains an infinite antichain of graphs.
Excluding these graphs one by one, we obtain an infinite strictly descending sequence of subclasses of $\cal X$, none of which is wqo.

To overcome this difficulty, we employ the notion of boundary classes, which can be viewed as a relaxation of the notion of minimal classes.
This notion was introduced in \cite{Ale03} to study the maximum independent set
problem in hereditary classes. Later, this notion was applied to some other graph
problems of both algorithmic \cite{AKL04,ABKL07,KLMT11,M13,M14,M16,MP16} and combinatorial \cite{boundary-wqo,Loz08} nature.
In particular, in \cite{boundary-wqo} it was applied to the study of well-quasi-ordered classes and was defined as follows.

To simplify the discussion, we use the term {\it bad} to refer to classes of graphs
that are not well-quasi-ordered by the induced subgraph relation and the term {\it good}
to refer to those classes that are well-quasi-ordered.

\begin{definition}
We say that $\cal X$ is a {\it limit class} if $\cal X$ is the intersection of any sequence
${\cal X}_1\supseteq {\cal X}_2\supseteq {\cal X}_3\supseteq\ldots$ of bad classes.
\end{definition}

In \cite{boundary-wqo}, it was shown that every bad class contains a minimal limit class.

\begin{definition}
A minimal limit class is called a {\it boundary class}.
\end{definition}

The importance of this notion is due to the following theorem, also proved in \cite{boundary-wqo}.

\begin{theorem}\label{thm:finite}
A  class of graphs defined by finitely many forbidden induced subgraphs is good if and only if it contains no boundary classes.
\end{theorem}

One of the boundary classes identified in \cite{boundary-wqo} is the class of linear forests.
It is the limit of the sequence $\Free(K_{1,3},C_3,\ldots,C_k)$ with $k$ tending to infinity.
Each class of this sequence is bad, since each of them contains infinitely many cycles.
However, the class of linear forests is not boundary in the universe of bipartite permutation graphs,
since none  of the classes $\Free(K_{1,3},C_3,\ldots,C_k)$ belongs to this universe.

In order to identify a boundary class $\cal X$ in the universe of  bipartite permutation graphs,
we need to construct a sequence of bad subclasses of bipartite permutation graphs converging to $\cal X$
and to show that $\cal X$ is a minimal limit class. To this end, we consider the following sequence:
$$
{\cal BP}\cap \Free(C_4,K_{1,4},S_{1,2,2},3K_{1,3},H_1,\ldots,H_k).
$$

We denote the limit class of this sequence with $k$ tending to infinity by $\cal L$.

\begin{lemma}\label{lem:X-structure}
$\cal L$ is the class of graphs, in which every connected component is a path, except possibly for two components of the form $S_{1,1,k}$ for some $k$.
\end{lemma}

\begin{proof}
By forbidding $C_4$ in the universe of  bipartite permutation graphs we restrict ourselves to caterpillar forests.
By forbidding  $K_{1,4}$ we further restrict ourselves to caterpillar forests of vertex degree at most $3$.
If, additionally, we forbid an $S_{1,2,2}$, then every connected graph becomes an $H_k$ or an $S_{1,1,k}$ or a $P_k$.
Since all graphs of the form $H_k$ are forbidden, every connected graph is either an $S_{1,1,k}$ or a $P_k$.
Finally, since  $3K_{1,3}$ is forbidden, at most two components have the form $S_{1,1,k}$.
\end{proof}

To show the minimality of the class $\cal L$, we use the following criterion proved in \cite{boundary-wqo}.

\begin{lemma}\label{lem:minimality-criterion}
A limit class ${\cal X}=\Free(\cal M)$ is minimal (i.e., boundary) if and only if
for every graph $G\in {\cal X}$ there is a finite set $\cal T\subseteq \cal M$,
such that $\Free(\{G\}\cup \cal T)$ is good.
\end{lemma}

\begin{theorem}
$\cal L$ is the only boundary class in the universe of bipartite permutation graphs.
\end{theorem}

\begin{proof}
Let $G$ be a graph in $\cal L$. According to Lemma~\ref{lem:X-structure}, there is a $k$ such that $G$ is an induced subgraph of $H_i$ for all $i\ge k$.
Therefore, $\Free(G,C_4,K_{1,4},S_{1,2,2},3K_{1,3},H_1,\ldots,H_{k-1})$ is a subclass of $\cal L$. It is well-known that graphs in a class are well-quasi-ordered
if and only if connected graphs in that class are well-quasi-ordered. Clearly, connected graphs in $\cal L$ are well-quasi-ordered,
and, hence, by Lemma~\ref{lem:minimality-criterion}, $\cal L$  is a boundary class in the universe of bipartite permutation graphs.

To prove the uniqueness of $\cal L$, we observe that the antichain of $H$-graphs is canonical in the universe of bipartite permutation graphs
in the sense that every hereditary subclass of bipartite permutation graphs containing finitely many $H$-graphs is well-quasi-ordered under induced
subgraphs \cite{canonical}. Therefore, by forbidding any graph from $\cal L$, we obtain a class which is well-quasi-ordered. This implies by Theorem~\ref{thm:finite} that
$\cal L$ is a unique boundary class in the universe of bipartite permutation graphs.
\end{proof}

\begin{theorem}\label{thm:classes-wqo}
Let $\Free(\mathcal L)$ denote the family of hereditary subclasses of bipartite permutation graphs, none of which contains $\cal L$ as a subclass.
Then $\Free(\mathcal L)$ is well-founded with respect to inclusion, i.e.,  it contains no strictly descending infinite chains of classes.
\end{theorem}

\begin{proof} 	
Note that any class $\mathcal X \in \Free(\mathcal L)$ is wqo, since it contains only finitely many $H$\nobreakdash -graphs. 
It is well-known (and not difficult to see) that a class of graphs is wqo if and only if the set of its subclasses is well-founded under inclusion. 
This immediately implies that $\Free(\mathcal L)$ itself is well-founded under inclusion.
\end{proof}

Theorem~\ref{thm:classes-wqo} shows that in the universe of bipartite permutation graphs the unique boundary class $\cal L$ is the only obstacle to finding minimal ``difficult'' classes.
Note, however, that the number of minimal classes may in principle be infinite: an example of an infinite antichain of classes with respect to inclusion in $\Free(\mathcal L)$ 
is given by the sequence $(\mathcal X_i)_{i \geq 1}$, with $\mathcal X_i := \mathcal{BP} \cap \Free(P_{4 + i}, H_1, \dots, H_{i - 1})$ 
(in particular, we have $\mathcal X_1 = \mathcal{BP}\cap \Free(P_5)$). We remark that the above antichain is constructed in such a way that $H_i \in \mathcal X_i \setminus \bigcup_{j \neq i} \mathcal X_j$. 
The dependence of the construction on the canonical antichain of $H$-graphs raises the following question.

\begin{problem}
Is the antichain $(\mathcal X_i)_{i \geq 1}$ canonical in $\mathcal \Free(\mathcal L)$?
\end{problem}

One more open problem related to the notion of well-quasi-ordering comes from the fact that the set of $H$-graphs, which forms a (canonical) antichain
with respect to induced subgraphs, is a chain with respect to induced minors, where an induced minor of a graph $G$ is any graph obtained from $G$ by 
a (possibly empty) sequence of vertex deletions and edge contractions. 

Well-quasi-orderability under induced minors has been studied before -- see, for instance, \cite{induced-minor}, where it was shown that interval graphs are not wqo under the relation, 
but chordal graphs of bounded clique number are. We note that the class of bipartite permutation graphs is not closed under taking 
induced minors, but this should not stop us from investigating its orderability with respect to the induced minor relation, 
especially given its similarity with the class of unit interval graphs (see, e.g., \cite{canonical}). In particular, we ask the following question.


\begin{problem}
Is the class of bipartite permutation graphs well-quasi-ordered under the induced minor relation?
\end{problem}

\section{Graph parameters and minimal classes}
\label{sec:parameters}

Many important parameters are bounded in the class of bipartite permutation graphs, which is the case for the size of a maximum clique, chromatic number, contiguity, etc.
Many other parameters can be arbitrarily large in this class. In the present section, we characterize several such parameters in terms of minimal subclasses of
bipartite permutation graphs where these parameters are unbounded.
We start by reporting in Section~\ref{sec:known} some known results, which will be helpful in the subsequent sections. 

\subsection{Neighbourhood diversity, distinguishing number and uniformicity}
\label{sec:known}

The notion of {\it neighbourhood diversity} was introduced in \cite{Lampis} and  can be defined as follows.

\begin{definition}
We say that two vertices $x$ and $y$ are {\it similar} if there is no vertex $z \neq x, y$ distinguishing them,
i.e., if there is no vertex $z \neq x, y$ adjacent to exactly one of $x$ and $y$. Vertex similarity is an equivalence relation. We denote by
$nd(G)$ the number of similarity classes in $G$ and call it the {\it neighbourhood diversity} of $G$.
\end{definition}

Neighbourhood diversity was characterized in \cite{parameters} by means of nine minimal hereditary classes of graphs
where this parameter is unbounded. Three of these classes are subclasses of bipartite graphs: 
\begin{itemize}
\item the class ${\cal X}_1$ of chain graphs,
\item the class ${\cal X}_2$ of graphs of vertex degree at most $1$,
\item the class ${\cal X}_3$ of bipartite complements of graphs of vertex degree at most $1$, i.e.,
bipartite graphs, in which every vertex has at most one non-neighbour in the opposite part. 
\end{itemize}
Three other classes are complements of graphs in  ${\cal X}_1$, ${\cal X}_2$ and ${\cal X}_3$, 
and the remaining three classes are subclasses of split graphs obtained by creating a clique 
in one of the parts in a bipartition of graphs in  ${\cal X}_1$, ${\cal X}_2$ and ${\cal X}_3$.
The subclass of split graphs obtained in this way from graphs in ${\cal X}_1$ is known as {\it threshold graphs}.
Only two of the nine listed classes are subclasses of bipartite permutation graphs,
which allows us to make the following conclusion. 

\begin{theorem}\label{thm:nbd}
The classes of chain graphs and graphs of vertex degree at most $1$  are
the only two minimal hereditary subclasses of bipartite permutation graphs of unbounded neighbourhood diversity.
\end{theorem}

\medskip 
The notion of distinguishing number appeared implicitly in \cite{Jump} and was given its name in \cite{Bell}.
To define this notion, consider a graph $G$, a subset $U \subseteq V(G)$, a collection of pairwise disjoint subsets $U_1$,~\dots,~$U_m$ of~$V(G)$, also disjoint from $U$. 
We will say that  $U$~\emph{distinguishes} the sets $U_1$, $U_2$,~\dots,~$U_m$  if for each~$i$, all vertices of~$U_i$
have the same neighbourhood in $U$, and for each $i \neq j$,
vertices $x \in U_i$ and $y \in U_j$ have different neighbourhoods in~$U$.

\begin{definition}
The {\it distinguishing number} of $G$ is the maximum $k$ such that $G$ contains a subset $U\subset V(G)$ that 
distinguishes at least $k$ subsets of $V(G)$, each of size at least $k$.
\end{definition}

The paper~\cite{Jump} provides a complete description of minimal classes of unbounded distinguishing number, of which there are precisely 13:

\begin{itemize}
\item the class of graphs every connected component of which is a clique,
\item the class of chain graphs,
\item the class of threshold graphs,
\item the class of star forests,
\item the class of graphs obtained from star forests by creating a clique on the leaves of the stars,
\item the class of graphs obtained from star forests by creating a clique on the centers of the stars,
\item the class of graphs obtained from star forests by creating a clique on the leaves of the stars and a clique on the centers of the stars,
\item the classes of complements of graphs in the above listed classes (note that the complements of threshold graphs are threshold graphs).
\end{itemize}

The global structure of graphs of bounded distinguishing number can informally be described as follows: the vertex set of every graph admits a partition into finitely many subsets
such that each subset induces a graph of bounded degree or co-degree (i.e., degree in the complement) and the edges between any two subsets form a bipartite graph of bounded degree or co-degree.
Note that bounded neighbourhood diversity is the special case where we require the degree or co-degree to equal 0, so that bounded neighbourhood diversity implies bounded distinguishing number (the same conclusion 
can be obtained by comparing the respective lists of minimal classes).

\medskip
One can define a parameter between neighbourhood diversity and distinguishing number, known as {\it uniformicity}, as follows.  
Let $k$ be a natural number,  $F$ 
a simple graph on the vertex set $\{1, 2, \ldots, k\}$, and $K$ a graph on $\{1, 2, \ldots, k\}$ with loops allowed. 
Let $U(F)$ be the disjoint union of infinitely many copies of $F$, and for $i = 1, \ldots,k$, let $V_i$ be the subset of
vertices of $U(F)$ containing vertex $i$ from each copy of $F$. 
Now we construct from $U(F)$ an infinite graph $U(F,K)$ on the same vertex set
by connecting two vertices $u \in V_i$ and $v \in V_j$ if and only if 
\begin{itemize}
\item either $uv \in E(U(F))$ and $ij\not\in E(K)$,
\item or $uv \not\in E(U(F))$ and $ij\in E(K)$.
\end{itemize}

Intuitively, we start with independent sets $V_i$ corresponding to the vertices of $F$, and for $i \neq j$, 
the sets $V_i$ and $V_j$ have an induced matching between them wherever $F$ has an edge. 
We then apply complementations according to the edges of $K$: if $ij\in E(K)$, we complement the edges between $V_i$ and $V_j$ (replacing $V_i$ with a clique if $i = j$). 

Finally, we let ${\cal P}(F,K)$ be the hereditary class consisting of all the finite induced subgraphs of $U(F,K)$.    
As an example, if $k=2$, $F=K_2$ and $E(K)=\emptyset$, then ${\cal P}(F,K)$ is the class of graphs of vertex degree at most $1$, 
and if $E(K)=\{12\}$ instead, then ${\cal P}(F,K)$ is the class of bipartite complements of graphs of vertex degree at most $1$.

\begin{definition}
A graph $G$ is called $k$-{\it uniform} if there is a number $k$ such that $G \in {\cal P}(F,K)$ for some $F$ and $K$.
The minimum $k$ such that $G$ is $k$-uniform is the {\it uniformicity} of $G$.
\end{definition}

From the results in  \cite{SpHerProp} and \cite{Jump} it follows that classes of bounded uniformicity are precisely the classes whose speed (i.e., the number of $n$-vertex labelled graphs) is below the Bell number. 
The set of minimal hereditary classes with speed at least the Bell number, and hence of unbounded uniformicity, consists of the 13 minimal classes of unbounded distinguishing number together with 
an infinite collection of classes, which can be described as follows. 

Let $A$ be a finite alphabet, $w=w_1w_2\ldots$ an infinite word over $A$, and $K$ an undirected graph with loops allowed and with vertex set $V(K)=A$.
We define the graph $U(w,K)$ as follows: the vertex set of $U(w,K)$ is the set of natural numbers and the edge set consists of pairs of distinct numbers $i,j$
such that
\begin{itemize}
\item either $|i-j|=1$ and $w_{i}w_{j} \notin{E(K)}$,
\item or $|i-j| > 1$ and $w_{i}w_{j} \in E(K)$.
\end{itemize}
To illustrate this notion, consider the case of $A=\{a\}$ and $E(K)=\emptyset$. 
Then the infinite word $w=aaa\ldots$ defines the infinite path $U(w,K)$. In fact, if $E(K)=\emptyset$, then $U(w,K)$ is an infinite path for any $w$.

Define ${\cal P}(w, K)$ to be the hereditary class consisting of all 
the finite induced subgraphs of $U(w,K)$. In particular, if $E(K)=\emptyset$, then ${\cal P}(w, K)$ is the class of linear forests.

A {\it factor} in a word $w$ is a contiguous subword, i.e., a subword whose letters appear consecutively in $w$.
A word~$w$ is called \emph{almost periodic} if for any factor~$f$
of~$w$ there is a constant~$k_f$ such that any factor of~$w$
of size at least~$k_f$ contains~$f$ as a factor.

\begin{theorem}{\rm \cite{Bell}}\label{thm:minimal-bell}
Let $X$ be a class of graphs with finite distinguishing number. Then $X$ is a minimal hereditary
class with speed at least the Bell number if and only if there exists a finite graph $K$ with loops allowed
and an infinite almost periodic word $w$ over $V(K)$ such that $X ={\cal P}(w,K)$.
\end{theorem}

\begin{definition}
Any class of the form ${\cal P}(w,K)$, where  $K$ is a finite graph with loops allowed
and $w$  is an infinite almost periodic word over $V(K)$, will be called a {\it class of folded linear forests}.
\end{definition}

We observe that each class of folded linear forests is a boundary class for well-quasi-orderability by induced subgraphs,
and in the world of classes with finite distinguishing number there are no other boundary classes, which was shown in \cite{finite}.

Summarising the discussion, we make the following conclusion.

\begin{theorem} 
The list of minimal hereditary classes of unbounded uniformicity consists of the 13 classes of unbounded distinguishing number and all the classes of 
folded linear forests.
\end{theorem}

With the restriction to bipartite permutation graphs, distinguishing number and uniformicity can be characterized in terms minimal classes as follows.

\begin{theorem} 
In the universe of bipartite permutation graphs, the list of minimal hereditary subclasses of unbounded distinguishing number consists of star forests and chain graphs, 
and the list of minimal hereditary subclasses of unbounded uniformicity consists of star forests, chain graphs and linear forests. 
\end{theorem}

\subsection{Lettericity and Parikh word representability}
\label{sec:lettericity}

Parikh word representable graphs were introduced in \cite{Parikh} as follows. Let $A=\{a_1<a_2<\ldots <a_k\}$ be an ordered alphabet,
and let $w=w_1w_2\ldots w_n$ be a word over $A$. The \emph{Parikh graph} of $w$ has $\{1,2,\ldots,n\}$ as its vertex set
and two vertices $i < j$ are adjacent if and only if there is $p\in \{1,2,\ldots,k-1\}$, such that $w_i=a_p$ and $w_j=a_{p+1}$.

In \cite{Parikh-bpg}, it was shown that the class of Parikh word representable graphs coincides with the class of bipartite permutation graphs
and was conjectured that every bipartite permutation graph with $n$ vertices admits a Parikh word representation over an alphabet of $\lfloor \frac{n}{2}\rfloor +1$ letters.
In this section, we prove the conjecture. To this end, we use the related notion of letter graphs, which was introduced in \cite{Pet02} and can be defined as follows.

As before, $A$ is a finite alphabet, but this time we do not assume any order on the letters of $A$. 
Let $D$  be a subset of $A^2$ and $w=w_1w_2\ldots w_n$ a word over $A$. 
The \emph{letter graph} $G(D,w)$ associated with $w$ has $\{1,2,\ldots,n\}$ as its vertex set, 
and two vertices $i < j$ are adjacent if and only if the ordered pair $(w_i,w_j)$ belongs to $D$. 
A graph $G$ is said to be a \emph{letter graph} if there exist an alphabet $A$, a subset $D\subseteq A^2$ and a word $w=w_1w_2\ldots w_n$ over $A$ such that $G$ is isomorphic to $G(D,w)$.

The role of $D$ is to decode (transform) a word into a graph and therefore we refer to $D$ as a decoder.
Every graph $G$ is trivially a letter graph over the alphabet $A = V(G)$ with the decoder $D = \{(v, w), (w, v): \{v, w\} \in E(G)\}$.
The \emph{lettericity} of $G$, denoted by $\ell(G)$, is the minimum $k$ such that $G$ is representable as a letter graph over an alphabet of $k$ letters.

To give a less trivial example, consider the alphabet $A=\{a,b\}$ and the decoder $D=\{(a,b)\}$.  Then, the word $ababababab$
describes the graph represented in Figure~\ref{fig:T5}. Clearly, this is a chain graph. Moreover, a chain graph of this form with $n$ vertices
in each part, which we denote by $Z_n$, is $n$-universal, i.e., it contains every $n$-vertex chain graph as an induced subgraph, as was shown in \cite{universal}.
Therefore a graph is a chain graph if and only if it is a letter graph over the alphabet $A=\{a,b\}$ with the decoder $D=\{(a,b)\}$.
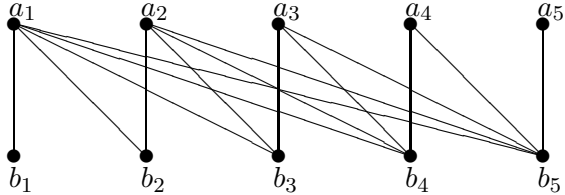
\begin{figure}[ht]
\begin{center}
\begin{picture}(300,60)
\put(50,0){\circle*{5}}
\put(100,0){\circle*{5}}
\put(150,0){\circle*{5}}
\put(200,0){\circle*{5}}
\put(250,0){\circle*{5}}
\put(48,53){$a_1$}
\put(98,53){$a_2$}
\put(148,53){$a_3$}
\put(198,53){$a_4$}
\put(248,53){$a_5$}



\put(48,-12){$b_1$}
\put(98,-12){$b_2$}
\put(148,-12){$b_3$}
\put(198,-12){$b_4$}
\put(248,-12){$b_5$}
\put(50,50){\circle*{5}}
\put(100,50){\circle*{5}}
\put(150,50){\circle*{5}}
\put(200,50){\circle*{5}}
\put(250,50){\circle*{5}}

\put(50,0){\line(0,1){50}}
\put(250,0){\line(-1,1){50}}
\put(250,0){\line(-2,1){100}}
\put(250,0){\line(-3,1){150}}
\put(250,0){\line(-4,1){200}}
\put(100,0){\line(0,1){50}}
\put(200,0){\line(-1,1){50}}
\put(200,0){\line(-2,1){100}}
\put(200,0){\line(-3,1){150}}
\put(150,0){\line(0,1){50}}
\put(150,0){\line(-1,1){50}}
\put(150,0){\line(-2,1){100}}
\put(200,0){\line(0,1){50}}
\put(100,0){\line(-1,1){50}}
\put(250,0){\line(0,1){50}}

\end{picture}
\end{center}
\caption{The letter graph of the word $ababababab$.
We use indices to indicate in which order the $a$-letters and the $b$-letters appear in the word.}
\label{fig:T5}
\end{figure}

The $n$-universal bipartite permutation graph $H_{n,n}$ (Figure~\ref{fig:H55}) can be viewed as a sequence 
of $n$ copies of the $n$-universal chain graphs and hence it
can be expressed as a letter graph on letters $a_1, \dots, a_n$ with decoder $\{(a_i, a_{i + 1}): 1 \leq i \leq n - 1\}$. The word representing $H_{n,n}$ consists of
the concatenation of $n$ copies of $a_1a_2\ldots a_n$. The same word represents $H_{n,n}$ as a Parikh graph over the ordered alphabet $A=\{a_1<a_2<\ldots <a_n\}$.
This discussion provides an alternative proof of the fact that the class of Parikh word representable graphs is precisely the class of bipartite permutation graphs.

\begin{theorem}\label{thm:lettericity-bound}
Let $G$ be a bipartite permutation graph with $n$ vertices. Then $G$ has lettericity bounded above by $\lfloor\frac{n}{2}\rfloor + 1$.
\end{theorem}
	
\begin{proof}
We first deal with the case when $G$ is connected, and assume $n \geq 2$.
We know that $G$ can be expressed as a letter graph on letters $a_1, \dots, a_n$ with decoder $\{(a_i, a_{i + 1}): 1 \leq i \leq n - 1\}$.
		
Among all expressions $w_1w_2\dots w_n$ with that decoder, writing $l(j)$ for the index of the letter in position $j$ of the word,
pick one that minimises $\sum\limits_{j = 1}^n l(j)$ (i.e., an expression that minimises the sum of the indices of the letters in $w$).
Let us state some properties of this expression:
\begin{itemize}
\item $a_1$ appears at least once in $w$. If not, we can shift all indices down by 1.
\item The last letter is $a_2$. Indeed, the last letter cannot be $a_1$, since that would mean $G$ has an isolated vertex.
If the last letter is $a_t$, for some $t \geq 3$, we can remove it, and add an $a_{t - 2}$ at the beginning of $w$: this new expression still represents $G$, but the sum of indices is smaller.
\item Let $t \geq 2$, and let $w_j = a_t$ be the rightmost appearance of $a_t$ in $w$.
Then there is an $a_{t - 1}$ to the right of $w_j$ in $w$.
Otherwise, as before, we can (by going more to the right if necessary) find an $a_t$ to the right of every $a_{t - 1}$ and of every $a_{t + 1}$,
which can be replaced by an $a_{t - 2}$ in the beginning of the word.
It follows that the rightmost occurrence of $a_t$ has to its right at least one $a_i$ for each $2 \leq i \leq t$, and no $a_i$ for $i > t$.
\item Let $t \geq 2$, and let $w_j = a_t$ be the rightmost occurrence of $a_t$.
Then, there is an $a_{t - 1}$ to the left of $a_t$.
Indeed, since $G$ is connected, the vertex $j$ has a neighbour.
But as seen above, there are no $a_{t + 1}$s to the right of $w_j$, hence that neighbour must be an $a_{i - 1}$ to its left.
\end{itemize}
	
Putting $r := \max_{1 \leq j \leq n} l(j)$, the above discussion implies that $w$ uses letters $2, \dots, r - 1$ at least twice, and letters $1, r$ at least once.
Since $G$ has $n$ vertices, this implies $r \leq \frac{n}{2} + 1$.
		
If $G$ is disconnected, writing $G_i$, $1 \leq i \leq s$ for its connected components,
we can produce words $w(G_i)$ as above for each $G_i$, where $G_1$ uses letters 1 to $a_{r_1}$, $G_2$ uses letters $a_{r_1}$ to $a_{r_2}$, and so on.
We then obtain a word representing $G$ by concatenating $w(G_s), w(G_{s - 1}), \dots, w(G_1)$ in that order.
The resulting word uses once more each letter twice, except for possibly the first and the last one.
\end{proof}

The upper bound in Theorem~\ref{thm:lettericity-bound} is tight and attained on graphs of vertex degree at most 1 (see \cite{Parikh-bpg} for arguments
given in the terminology of Parikh word representability or \cite{cographs} for arguments given in the terminology of lettericity). In the rest of
this section we show that, within the universe of bipartite permutation graphs, the class of graphs of vertex degree at most 1 is the only obstacle
for bounded lettericity, i.e., it is the unique minimal subclass of bipartite permutation graphs of unbounded lettericity.

\begin{theorem}
For each $p$, there is an $f(p)$ such that the lettericity of $pK_2$-free bipartite permutation graphs is at most $f(p)$.
\end{theorem}

\begin{proof}
Let $G$ be a $pK_2$-free bipartite permutation graph. Then $G$ has at most $p-1$ nontrivial connected components, i.e., components of
size at least 2. Each component can be embedded, as an induced subgraph, into the universal graph $H_{n,n}$ with at most $3p-2$ rows,
since any connected induced subgraph of the universal graph occupying at least $3p-1$ rows contains an induced $P_{3p-1}$ and, hence,
an induced $pK_2$. Therefore, any component of $G$ requires at most $3p-2$ letters to represent it. Altogether, we need at most $(p-1)(3p-2)+1$
letters to represent $G$.
\end{proof}

\subsection{Shrub-depth and related parameters}
\label{sec:shrub}

In this section, we analyse various width and depth parameters and characterize them in terms of minimal hereditary
classes of bipartite permutation graphs where these parameters are unbounded.

We start by repeating that the class of all bipartite permutation graphs is a minimal hereditary class of unbounded clique-width.
{\it Tree-width} is a restriction of clique-width in the sense that bounded tree-width implies bounded clique-width, but not necessarily vice versa.

\begin{proposition}
The class of complete bipartite graphs is the only minimal hereditary subclass of bipartite permutation graphs of unbounded tree-width.
\end{proposition}

\begin{proof}
It is well-known that tree-width can be arbitrarily large for complete bipartite graphs. On the other hand,
it was shown in \cite{cw-bound} that for graphs that do not contain $K_{n,n}$ as a subgraph, tree-width is bounded by a function of $n$ and its clique-width. 
Therefore, for every subclass of bipartite permutation
graphs excluding at least one complete bipartite graph tree-width is bounded, since clique-width is bounded.
\end{proof}

{\it Tree-depth} is a restriction of tree-width in the sense that bounded tree-depth implies bounded tree-width, but not necessarily vice versa.
Let us define the {\it path number} of $G$ to be the length of a longest (not necessarily induced) path in $G$.
In the terminology of minimal classes, the two parameters are equivalent, as shown in the following proposition.

\begin{proposition}
The classes of complete graphs, complete bipartite graphs and linear forests are the only three minimal hereditary classes of graphs of unbounded tree-depth and path number.
In the universe of bipartite permutation graphs, complete bipartite graphs and linear forests are the only two minimal hereditary classes of graphs of unbounded tree-depth and path number.
\end{proposition}

\begin{proof}
It is known \cite{tree-depth} that tree-depth is unbounded in a class $\cal X$ if and only if graphs in $\cal X$ contain arbitrarily long paths as subgraphs.
Also, it was shown in \cite{Razgon} that for every $t, p, s$, there exists a $z=z(t,p,s)$,
such that every graph with a (not necessarily induced) path of length at least $z$ contains either an induced path of length $t$
or an induced complete bipartite graph with color classes of size $p$ or a clique of size $s$.
The proposition then follows.
\end{proof}

{\it Shrub-depth} is an extension of tree-depth for dense graphs (see \cite{shrub} for the original definition). {\it Rank-depth} is a related parameter, which is equivalent to shrub-depth in the following sense.

\begin{theorem}{\rm \cite{rank-depth}}
A class of graphs has bounded rank-depth if and only if it has bounded shrub-depth.
\end{theorem}

Our next result characterizes both parameters in the terminology of minimal classes within the universe of bipartite permutation graphs.

\begin{theorem}\label{thm:shrub-depth}
The classes of chain graphs and linear forests are the only two minimal hereditary subclasses of bipartite permutation graphs of unbounded shrub-depth and rank-depth.
\end{theorem}

\begin{proof}
It was shown in \cite{rank-depth} that paths, and therefore linear forests, have unbounded rank-depth. 


To show that the parameters are unbounded in the class of chain graphs, we use the notions of local complementations and local equivalence. 
A {\em local complementation} is the operation of complementing the subgraph induced by the neighbourhood of a vertex, 
and two graphs $G$ and $H$ are said to be {\em locally equivalent} if $G$ can be obtained from $H$ by a sequence of local complementations (and vice-versa). 
It was also pointed out in \cite{rank-depth} that, as a consequence of a result from \cite{localcomp}, rank-depth is invariant under local complementations, hence it suffices to show that there are arbitrarily 
long paths that are locally equivalent to chain graphs.

There is one specific sequence of local complementations known as a pivot.
Applied to an edge $uv$ of a graph, the pivot consists of complementing the neighbourhoods of $u$, $v$ and then $u$ again. 
Its net effect is to complement the edges between $N(u)\setminus\{v\}$ and $N(v)\setminus\{u\}$; if the starting graph is bipartite, it remains so after the pivot. 

We observe that the universal chain graph $Z_n$ (Figure~\ref{fig:T5}) can be pivoted to a path $P_{2n}$ and vice versa.
To transform $Z_n$ into $P_{2n}$, one can apply pivoting on the edges $a_2b_2,a_3b_3,\ldots,a_{n-1}b_{n-1}$.  
Therefore, rank- and shrub-depth are unbounded in the class of chain graphs as claimed.

\medskip
It remains to show that by excluding a path $P_k$ and a chain graph $Z_t$ within the universe of bipartite permutation graphs, we obtain a class $\cal X$ of bounded rank- and shrub-depth.
Immediately from the definition, we note that bounded neighbourhood diversity implies bounded shrub-depth, and that it is enough to prove that shrub-depth is bounded for connected graphs in $\mathcal X$.  

Let $G$ be a connected graph in $\mathcal X$.  
Since $P_k$ is forbidden, $G$ has lettericity at most $k - 1$, i.e., it can be embedded into the universal construction $H_{n,n}$ using at most $k-1$ consecutive rows. 
Every two consecutive rows induce a chain graph, and since $Z_t$ is forbidden, this chain graph has neighbourhood diversity bounded by a constant, by Theorem \ref{thm:nbd}. 
It follows that the neighbourhood diversity of $G$ is bounded, as required. 
\end{proof}

Outside of the universe of bipartite permutation graphs there exist other minimal classes of unbounded shrub-depth and rank-depth related to linear forests and chain graphs, 
such as classes of folded linear forests (defined in Section~\ref{sec:known}), the class of complements of chain graphs and the class of threshold graphs.
Theorem~\ref{thm:shrub-depth} suggests the following conjecture.

\begin{conjecture}
Shrub-depth and rank-depth are unbounded in a hereditary class $\cal X$ if and only if $\cal X$ contains a minimal hereditary class of unbounded shrub-depth and rank-depth.
The set of minimal classes is infinite and consists of all the classes of folded linear forests,  as well as  the classes of chain graphs, 
complements of chain graphs and threshold graphs. 
\end{conjecture}


\section{Algorithmic problems and minimal classes}
\label{sec:algo}

The simple structure of bipartite permutation graphs allows efficient solutions for many algorithmic problems that are generally NP-hard.
However, some NP-hard problems remain intractable in this class, which is the case, for instance, for {\sc pair-complete coloring} (also known as
{\sc achromatic number}), {\sc harmonious coloring} \cite{harmonious}, and {\sc induced subgraph isomorphism} \cite{isi}. In this section, we focus on the latter
of these problems and identify a minimal hereditary subclass of bipartite permutation graphs where the problem is NP-hard.

The {\sc induced subgraph isomorphism} ({\sc ISI}, for short) problem can be stated as follows: given two graphs $H$ and $G$, decide whether $H$
is an induced subgraph of $G$ or not. A hereditary class $\cal X$ will be called \emph{{\sc ISI}-easy} if, for every instance $(G, H)$ with $G \in \cal X$, the problem can be solved in time
polynomial in  $|V(G)|$, where the polynomial is independent of $|V(H)|$. Otherwise, $\cal X$  will be called \emph{{\sc ISI}-hard}.
An inclusion-wise minimal {\sc ISI}-hard class will be called \emph{minimal {\sc ISI}-hard}.

The {\sc induced subgraph isomorphism} is known to be NP-complete in the class of linear forests, which is a proper subclass of bipartite permutation graphs \cite{isi}. 
In this section, we show that the class of linear forests is a minimal {\sc ISI}-hard hereditary class. To do so, we express the ISI problem in proper subclasses of linear forests as a linear programming problem.

\bigskip

We represent a $P_{n+1}$-free linear forest as $\alpha_1P_1+\alpha_2P_2+\ldots+\alpha_{n-1}P_{n-1}+\alpha_nP_n$, 
where $\alpha_i \in \mathbb N_{\geq 0}$, and the forest contains exactly $\alpha_i$ connected components isomorphic to $P_i$.

If $G$ is a $P_{n+1}$-free linear forest, then all its induced subgraphs are also $P_{n+1}$-free linear forests, and the recognition problem for the class of $P_{n+1}$-free linear forests can be solved in polynomial time. So, we may assume that
$$H=\alpha_1P_1+\alpha_2P_2+\ldots+\alpha_{n-1}P_{n-1}+\alpha_nP_n,~\text{and}$$
$$G=\beta_1P_1+\beta_2P_2+\ldots+\beta_{n-1}P_{n-1}+\beta_nP_n.$$

Now, for each $1 \leq i \leq n$ we construct an integer matrix $A_i$ whose columns correspond to induced subgraphs of the path $P_i$. Specifically, $A_i$ has a column $(\gamma_1, \gamma_2, \dots, \gamma_n)$ for each linear forest $\gamma_1 P_1 + \gamma_2 P_2 + \dots + \gamma_n P_n$ that is an induced subgraph of the path $P_i$. Write $m_i$ for the number of columns of $A_i$, and let $A = (a_{ij}) := (A_1|A_2|\ldots|A_n)$ be the horizontal concatenation of the $A_i$, with $m := m_1 + \dots + m_n$ columns.



Finally, consider the following system of linear constraints:
\begin{equation}
\label{s1}
\begin{cases}
\sum\limits_{j=1}^{m}a_{ij}x_j\geq \alpha_i,\forall i\in \{1,2,\ldots,n\},\\
\sum\limits_{j=m_1+\ldots+m_{i-1}+1}^{m_1+\ldots+m_{i}}x_j\leq \beta_i,\forall i\in \{1,2,\ldots,n\}.
\end{cases}
\end{equation}

\begin{lemma}
	\label{l1a}
	The graph $H$ is an induced subgraph of $G$ if and only if the system (\ref{s1}) is compatible.
\end{lemma}

\begin{proof} 
	We interpret each column vector $(x_1, \dots, x_m)^T \in \mathbb{N}^m$ as an attempt to embed, for each $j$, $x_j$ copies of the forest $a_{1,j}P_1+\ldots+a_{n,j}P_n$ into different components of $G$. A sufficient condition for a simultaneous embedding to exist where each copy (across all $j$) uses a different component of $G$ is that, for each $i \in \{1, \dots, n\}$,
	$$\sum\limits_{j=m_1+\ldots+m_{i-1}+1}^{m_1+\ldots+m_{i}}x_j\leq \beta_i.$$
	
	Note that, if such an embedding exists, $G$ must contain the forest $$\sum\limits_{j=1}^{m}a_{1j}x_jP_1+\sum\limits_{j=1}^{m}a_{2j}x_jP_2+\ldots+\sum\limits_{j=1}^{m}a_{nj}x_jP_n.$$ If, in addition,  we have for each $i \in \{1, \dots, n\}$ $$\sum\limits_{j=1}^{m}a_{ij}x_j\geq \alpha_i,$$ then the above forest, and hence $G$, must contain $H$ as an induced subgraph. This shows that a feasible solution to the system implies the existence of an embedding of $H$ into $G$.
	
	\medskip
	
	Conversely, assume that $H$ has an induced subgraph embedding $\iota : H \hookrightarrow G$. To produce a feasible solution to the system, write $G_1, \dots, G_t$ for the connected components of $G$. For $1 \leq s \leq t$, let $v_s$ be the standard basis (column) vector in $\mathbb{N}^m$ described as follows: if $G_s$ is isomorphic to $P_i$, then $v_s$ has a 1 in the row between $m_1 + \dots + m_{i - 1} + 1$ and $m_1 + \dots + m_i$ corresponding to the forest induced by $\iota(H) \cap G_s$. 
	
	Now let $x := \sum_{s = 1}^t v_s$. It can be checked that, by construction, $x$ is a feasible solution (with equality for the constraints involving the $\alpha_i$).
%
%
%
\end{proof}

\begin{lemma}
\label{l2a}
For any fixed $n$, the {\sc ISI} problem can be solved in polynomial time for $P_{n+1}$-free linear forests.
\end{lemma}

\begin{proof}
Clearly, $m_i\leq 2^i$, for any $i$. Therefore, $m\leq n2^n$. The integer linear programming problem is known to be polynomial for a fixed number of variables \cite{D12,GM19,K87,L83,VGZC19}. 
More precisely, its best known complexity bound for $k$ variables is $O(k^k)$, multiplied by a polynomial in the length of the input data \cite{D12,GM19,VGZC19}. 
By these facts and Lemma~\ref{l1a}, the result holds.
\end{proof}

The following statement is the main result of this section.

\begin{theorem}
\label{t1}
The class of linear forests is minimal hard for the {\sc ISI} problem.
\end{theorem}

\begin{proof}
The class of all linear forests is known to be hard for the {\sc ISI} problem \cite{isi}.
For any proper hereditary subclass $\cal X$ of  this class, a linear forest $F$ is forbidden.
If $F$ has $n$ vertices, then $F$ is an induced subgraph of $P_{2n}$
and, hence, all graphs in $\cal X$ are $P_{2n}$-free.
The theorem then follows from Lemma~\ref{l2a}.
\end{proof}

If there exist {\sc ISI}-hard subclasses of bipartite permutation graphs not containing the class of linear forests, then,
according to Theorem~\ref{thm:classes-wqo}, each such subclass contains a minimal {\sc ISI}-hard class. In other words,
in the universe of bipartite permutation graphs the {\sc ISI} problem can be characterized by a set of minimal {\sc ISI}-hard classes.
We believe that this set consists of a single class.

\begin{conjecture}
Unless P=NP, the {\sc induced subgraph isomorphism} problem is NP-hard in a hereditary subclass $\cal X$ of bipartite permutation graphs
if and only if $\cal X$ contains all linear forests. Equivalently, for any fixed $k$, the problem can be solved in polynomial time for $P_k$-free bipartite permutation graphs.
\end{conjecture}

To support the conjecture, we solve the problem for $P_5$-free bipartite graphs. 

\begin{proposition}
If both $G$ and $H$ are $P_5$-free bipartite graphs, then the {\sc induced subgraph isomorphism} problem
can be solved for the pair $(G,H)$ in polynomial time. 
\end{proposition}

\begin{proof}
We reduce the problem to finding a maximum weight one-sided-perfect matching in an auxiliary edge-weighted complete bipartite graph $B$ with order linear in $|V_G|$. 
This is equivalent to the assignment problem, which is well-known to take polynomial time. We first describe the graph $B$, then show how it can be constructed in polynomial time.
	
\smallskip

Note that every connected component $G'$ of $G$ can accommodate at most one 
non-trivial 
component $H'$ of $H$ as an induced subgraph, and potentially some isolated vertices of $H$. 
The graph $B = (V_G \cup V_H, E = V_G \times V_H, \omega:E \to \mathbb Z \cup \{-\infty\})$ is defined as follows: 
$V_G$ represents non-trivial connected components of $G$, while $V_H$ represents non-trivial connected components of $H$. The weight $\omega((G', H'))$ of 
the edge between components $G'$ and $H'$ indicates the maximum number of isolated vertices that can be accommodated by $G'$ in addition to $H'$, 
with $-\infty$ if $H'$ cannot be embedded into $G'$\footnote{As is usual in these situations, $-\infty$ is to be replaced, 
in practice, with a large enough negative number, say $-(|V(G)||V_G||V_H| + 1)$, that guarantees the weight of an optimal matching is negative if $H$ does not embed into $G$. 
For notational reasons, we will keep it as $-\infty$.}. 
With this set-up, one easily checks that $H$ has an induced embedding in $G$ if and only if $B$ has a matching of size $|V_H|$ whose weight is at least the number of isolated vertices of $H$.  

\smallskip

It remains to show that $B$ can be constructed in polynomial time. The only non-trivial part is determining the edge weights. 
In other words, for each of the $O(|V(G)||V(H)|)$ pairs $(G', H')$ of connected components, we must determine in polynomial time whether $H'$ can be embedded into $G'$, 
and if yes, we must find the maximum number of isolated vertices that can be embedded into $G'$ in addition to $H'$. 

Given an embedding $\iota : H' \hookrightarrow G'$, write $\mu(\iota)$ for the size of a maximum independent set in $G' \setminus N[\iota(H')]$, 
where $N[A]$ denotes the closed neighbourhood of a set $A$ (i.e., $A$ together with all vertices having at least a neighbour in $A$). 
Our edge weight $\omega((G', H'))$ is thus $\max\limits_{\iota : H' \hookrightarrow G'} (\mu(\iota))$ (with $-\infty$ for a maximum over an empty set). 

To compute the edge weights, we note that any connected $P_5$-free bipartite graph is $2K_2$-free (i.e., a chain graph), 
so we may encode it as a letter graph on the alphabet $\{a, b\}$ with decoder $\{(a, b)\}$ (see Section~\ref{sec:lettericity} for details). 
Any connected chain graph admits at most two such representations (depending on which letter represents which side). 
Given components $G'$ and $H'$, write $w_{H'}^1$ and $w_{H'}^2$ for the two words representing $H'$, and write $w_{G'}$ for one of the two words representing $G'$. 
This gives a one to one correspondence between the set of induced subgraph embeddings $\iota : H' \hookrightarrow G'$ and 
the union of the sets of subword embeddings $\iota_w : w_{H'}^i \hookrightarrow w_{G'}$ for $i = 1, 2$. 
In particular, $H'$ has an induced subgraph embedding $\iota$ into $G'$ if and only if one of $w_{H'}^1$ and $w_{H'}^2$ has a corresponding subword embedding $\iota_w$ into $w_{G'}$.

Determining whether such an embedding $\iota_w$ exists can be done greedily in linear time, but that only tells us whether the corresponding edge in $B$ has weight $-\infty$ or not. 
We claim that, for any embedding $\iota$, $\mu(\iota)$ can be easily computed from the corresponding subword embedding $\iota_w$. 
To see this, note that, since $H'$ is connected, the first letter of $w_{H'}$ is an $a$, while the last letter is a $b$. 
Then $\mu(\iota)$ is the number of $b$'s before the first $a$ in $\iota_w(w_{H'})$ plus the number of $a$'s after the last $b$. 
Indeed, those $b$s in the prefix $a$s in the suffix correspond to an independent set, and every other vertex is adjacent to at least one of the initial $a$ and final $b$ of $\iota_w(w_{H'})$. 

Thus $\mu(\iota)$ only depends on the positions of the first and last letters of the embedding $\iota_w$. This gives us the following way of determining $\omega((G', H'))$ in polynomial time:
\begin{enumerate}
	\item Set $\nu = -\infty$.

	\item Check (in linear time) if any of the two $w_{H'}^i$ embeds as a subword in $w_{G'}$. If not, return $\nu$.
	
	\item For each pair of letters $l_i, l_j$ in $w_{G'} = l_1l_2 \dots l_t$ with $i < j$, $l_i = a$ and $l_j = b$:
	
	\begin{enumerate}
		\item Write $w_{G'}$ as a concatenation $w_1w_2w_3$, where $w_2$ is the substring that starts with $l_i$ and ends with $l_j$. 
		
		\item Determine (in linear time) whether any of the two $w_{H'}^i$ embeds in $w_2$. 
If not, continue to the next pair. If yes, let $\nu'$ be the number of $b$s in $w_1$ plus the number of $a$s in $w_3$, and set $\nu := \max(\nu, \nu')$.
	\end{enumerate} 
	
	\item Return $\nu$.
\end{enumerate}

It is routine to check that the above procedure terminates in polynomial time, and that the value returned is $\max\limits_{\iota : H' \hookrightarrow G'} (\mu(\iota))$ as required.
\end{proof}

Let us repeat that the class of linear forests is one of the infinitely many classes ${\cal P}(w, K)$ of folded linear forests, defined in Section~\ref{sec:known}. 
We conjecture that all classes of this form are minimal hard for the {\sc induced subgraph isomorphism} problem. 

\begin{conjecture}
Each class of folded linear forests is a minimal hard class for the {\sc induced subgraph isomorphism} problem.
\end{conjecture}

\section{Universal graphs and minimal classes}
\label{sec:uni}

As we have seen earlier, the class of bipartite permutation graphs contains a universal element of quadratic order, i.e.,
a graph with $n^2$ vertices that contains all $n$-vertex bipartite permutation graphs as induced subgraphs. On the other
hand, for the class of chain graphs, we have an $n$-universal graph on $2n$ vertices, i.e., a universal graph of linear order. 
This raises many questions regarding the growth rates of order-optimal universal graphs for subclasses of bipartite permutation graphs. One of the most immediate questions is identifying a boundary separating classes with a universal graph of linear order from classes where the smallest universal graph is super-linear.
In this section, we show that the class of star forests is a minimal hereditary class with a super-linear universal graph.

Before we present the result for star forests, let us observe that in general not every hereditary class $\cal X$ contains a universal graph, 
and even if it does, an optimal universal construction for $\cal X$ does not necessarily belong to $\cal X$. 
In order to circumvent these difficulties (and to ensure downwards closure of the set of classes with, say, linear universal graphs), 
we will only consider universal constructions consisting of bipartite permutation graphs. In other words, whenever we look for universal constructions 
for some class $\mathcal X \subseteq \mathcal{BP}$, we allow {\em any} bipartite permutation graphs,  and {\em only} bipartite permutation graphs, to appear in our universal constructions.

\medskip

In the following lemma we first describe a star forest of order $O(n\cdot \log(n))$ containing all $n$-vertex star forests as induced subgraphs,
and then we show that this construction is (asymptotically) order-optimal in the universe of all bipartite permutation graphs. To simplify notation, throughout this section we denote the star $K_{1,n}$ by $S_{n}$.

\begin{lemma}\label{l2}
The minimum number of vertices in a bipartite permutation graph containing all $n$-vertex star forests is $\Theta(n\cdot \log(n))$.  
\end{lemma}

\begin{proof}
Let $F^*$ be the star forest $S_{\left\lfloor\frac{n}{1}\right\rfloor}+S_{\left\lfloor\frac{n}{2}\right\rfloor}+\ldots+S_{\left\lfloor\frac{n}{n-1}\right\rfloor}+S_{\left\lfloor\frac{n}{n}\right\rfloor}$.
It is a bipartite permutation graph, and it has $n$ connected components and $\sum\limits_{i=1}^{n} (\lfloor\frac{n}{i}\rfloor+1)$ vertices.
As $\lfloor x\rfloor<x$, for any $x$, $F^*$ has $O(\sum\limits_{i=1}^{n} \frac{n}{i})$ vertices. Recall that the $n$-th harmonic number $\sum\limits_{i=1}^{n} \frac{1}{i}$ 
is equal to	$\ln(n)+\gamma+\epsilon_n$, where $\gamma=0.577\ldots$ is the Euler--Mascheroni constant and $\epsilon_n$ tends to $0$ with $n$ tending to infinity.
Therefore, $F^*$ has $O(n\cdot \log(n))$ vertices.
	
Let us show that $F^*$ is a universal graph for $n$-vertex star forests. Indeed, let $F=S_{a_1}+S_{a_2}+\ldots+S_{a_p}$ be an $n$-vertex star forest,
where $a_1\geq a_2\geq \ldots \geq a_p$. Clearly, $i\cdot a_i\leq a_1+\ldots+a_i<n$, for any $1\leq i\leq p\leq n$.
Hence, $a_i<\frac{n}{i}$ and $a_i\leq \left\lfloor\frac{n}{i}\right\rfloor$, as $a_i$ is an integer. 
Therefore, for any $i$, $S_{a_i}$ is an induced subgraph of $S_{\left\lfloor\frac{n}{i}\right\rfloor}$. Thus, $F$ is an induced subgraph of $F^*$. 

\medskip	
To prove a lower bound, let $H$ be a bipartite permutation graph containing all $n$-vertex star forests. 
It can be embedded, as an induced subgraph, into $H_{n',n'}$  for some $n'$ (see Figure \ref{fig:H55}). 
Now let $n_1, n_2, \dots$ be a list in non-increasing order of the numbers of vertices of $H$ embedded in each row of $H_{n', n'}$ 
(so that $n_1$ is the number of vertices of $H$ in a row of $H_{n', n}$ with the most vertices of $H$, and so on).
We show that $n_i \geq \frac{1}{2}\left(\left\lfloor\frac{n}{10i}\right\rfloor-1\right)$ for any $1 \leq i \leq \left\lfloor\frac{n}{20} \right\rfloor$, 
implying that the graph $H$ has $\Omega(n\cdot \log(n))$ vertices.

Let $t \in \{10i: i \in \mathbb N\} \cap \{1, \dots, \left\lfloor\frac{n}{2}\right\rfloor\}$. 
By $F_t$, we denote the star forest with $t$ connected components, each isomorphic to $S_{\left\lfloor\frac{n}{t}\right\rfloor-1}$. 
For any $t$, the graph $F_t$ is an induced subgraph of $H$ and hence $F_t$ must embed into $H_{n',n'}$. 
Since any two consecutive rows of $H_{n',n'}$ induce a chain graph, i.e., a $2K_2$-free graph,  
any row of $H_{n',n'}$ contains the centres of at most two  stars of $F_t$. 
Each star intersects at most 3 consecutive rows of $H_{n', n}$, hence a star can intersect the same row as at most 9 other stars. It is therefore possible to find  $\frac{t}{10}$ stars
$S_{\left\lfloor\frac{n}{t}\right\rfloor-1}$ in $F_t$ such that no two of them intersect the same row of $H_{n',n'}$, and thus there are at least  $\frac{t}{10}$ 
pairwise distinct rows in $H_{n',n'}$, each of which contains at least $\frac{1}{2}\left(\left\lfloor\frac{n}{t}\right\rfloor-1\right)$ vertices of $H$. 
It follows that $n_{t/10} \geq \frac{1}{2}\left(\left\lfloor\frac{n}{t}\right\rfloor-1\right)$ or, 
changing indices, that $n_i \geq \frac{1}{2}\left(\left\lfloor\frac{n}{10i}\right\rfloor-1\right)$ for any $1 \leq i \leq \left\lfloor\frac{n}{20} \right\rfloor$ as required.
\end{proof}

\begin{theorem}
\label{l4}
The class of star forests is a minimal hereditary class that does not admit a universal bipartite permutation graph of linear order.
\end{theorem}

\begin{proof}
Let ${\mathcal X}$ be any proper hereditary subclass of the class $\mathcal{SF}$ of star forests. Then, ${\mathcal X}\subseteq \mathcal{SF}\cap \Free(kS_k)$
for some $k$. Therefore, every graph in $X$ consists of at most $k-1$ stars with at least $k$ leaves and arbitrarily many stars with at most $k-1$ leaves. 
But then $(k-1)S_n+nS_{k-1}$ is an $n$-universal graph for $\cal X$ of linear order.
\end{proof}

The class of star forests is not the only obstruction to admitting a universal graph of linear order.
To see this, we show that the class of $3S_6$-free bipartite permutation graphs requires a super-linear universal graph. 

\begin{lemma}\label{lem:univlowerbound}
	Suppose $H$ is a bipartite permutation graph containing all $n$-vertex $3S_6$-free graphs as induced subgraphs. Then $|V(H)| = \Omega(n^{3/2})$.
\end{lemma}

\begin{proof}
To prove the statement, we will show there must be $\Omega(n^3)$ pairs of vertices in $H$, from which we immediately get $$|V(H)|^2 \geq {|V(H)| \choose 2} \geq \Omega(n^3)$$ and so $|V(H)| = \Omega(n^{3/2})$.

\bigskip

We know $H$ can be embedded as an induced subgraph into a universal graph $H_{n', n'}$ for some $n'$. 
The main idea is to construct a structure that is ``rigid'', in the sense that we can guarantee the distance (within the structure) 
between certain vertices is not much greater than the distance in $H_{n', n'}$ between the embeddings of those vertices. 
To this end, we use a result due to Ferguson \cite{pathlet}, that states the lettericity of the path $P_s$, for $s \geq 3$, 
is precisely $\left\lfloor\frac{s + 4}{3}\right\rfloor$.\footnote{The bounds from \cite{Pet02} are sufficient for the proof, 
but make it a bit messier.} In our language, it follows that any embedding of a chordless path $P_s$ into $H_{n',n'}$ uses 
at least $\left\lfloor\frac{s + 4}{3}\right\rfloor$ layers (rows). 

Since edges appear only between successive layers, the set of layers used by an embedding of the path is an interval. 
Moreover, the ends of a chordless path do not appear more than one layer away from the extremal layers in the interval, 
otherwise a $2K_2$ forms between two of the layers. This implies that the distance between the ends of a chordless path $P_s$ 
must be at least $\left\lfloor\frac{s + 4}{3}\right\rfloor - 3 = \left\lfloor\frac{s - 5}{3}\right\rfloor$.

\bigskip

For each $t$, we construct in two steps a graph $Q_t$ as depicted in Figure \ref{fig:rigidgraph}.

\begin{figure}[ht]
	\centering
	\begin{subfigure}[t]{0.49\linewidth}
		\centering
		\begin{tikzpicture}[scale=1, transform shape]
		
		\foreach \i in {1,...,6}{
			\filldraw (\i, \i) circle (2pt) node[below left] {};
		}
		
		\foreach \i in {1,...,5}{
			\filldraw (\i + 1, \i) circle (2pt) node[below left] {};
		}
		
		\foreach \i in {1,...,5}{
			\filldraw (\i, \i + 1) circle (2pt) node[below left] {};
		}
		
		\foreach \i in {1,...,5}{
			\draw (\i, \i + 1) -- (\i + 1, \i);
			\draw (\i, \i) -- (\i, \i + 1);
			\draw (\i + 1, \i) -- (\i + 1, \i + 1);
		}
		
		\draw (1, 1) node[below left] {$q$};
		\draw (6, 6) node[above right] {$p$};

		\end{tikzpicture}
		\captionsetup{justification=centering}
		\caption*{Step 1: Start with a chordless path on \\ $3t + 7$ vertices.}
	\end{subfigure}	
	\begin{subfigure}[t]{0.49\linewidth}
		\centering
		\begin{tikzpicture}[scale=1, transform shape]
		
		\foreach \i in {0,...,5}{
			\filldraw (\i, \i) circle (2pt) node[below left] {};
		}
		
		\foreach \i in {0,...,4}{
			\filldraw (\i + 1, \i) circle (2pt) node[below left] {};
		}
		
		\foreach \i in {0,...,4}{
			\filldraw (\i, \i + 1) circle (2pt) node[below left] {};
		}
		
		\foreach \i in {0,...,4}{
			\draw (\i, \i + 1) -- (\i + 1, \i);
			\draw (\i, \i) -- (\i, \i + 1);
			\draw (\i + 1, \i) -- (\i + 1, \i + 1);
		}
		
		\foreach \i in {1,...,6}{
			\filldraw (-1, \i) circle (2pt) node[below left] {};
		}
		\draw (-1, 1) -- (-1, 6);
		
		\foreach \i in {2,...,6}{
			\draw (-1, \i) -- (\i - 2, \i - 1);	
		}
		
		\foreach \i in {1,...,6}{
			\draw (-1, \i) -- (\i - 1, \i - 1);
		}
		
		\foreach \i in {1,...,5}{
			\draw (-1, \i) -- (\i, \i - 1);
		}
		
		\draw (-1, 1) node[below left] {$y$};
		\draw (-1, 6) node[above left] {$x$};
		\draw (0, 0) node[below left] {$q$};
		\draw (5, 5) node[above right] {$p$};
		
		\end{tikzpicture}
		\captionsetup{justification=centering}
		\caption*{Step 2: Add a chordless path on $t + 3$ vertices, connecting it to the previous path as above.}	
	\end{subfigure}
	\caption{The graph $Q_t$ for $t = 3$}
	\label{fig:rigidgraph}
\end{figure}

It is easy to see that $Q_t$ is a bipartite permutation graph, since we can embed the rows in the figure into successive layers of the universal graph. 
Moreover, writing $d_G$ for the distance in a graph $G$, we have (using the triangle inequality and the above discussion, and assuming for now that $Q_t$ embeds into $H$), 
$$d_H(x, y) \geq d_H(p, q) - 2 \geq t - 2$$ and $$d_H(x, y) \leq d_{Q_t}(x, y) = t + 2.$$

\bigskip 

We now construct bipartite permutation graphs $R_{n, t}$ from $Q_t$ by replacing $x$ and $y$ with independent sets 
$X$, $Y$ of twins of size $\left\lfloor\frac{n - 4t - 8}{2}\right\rfloor$ each, with the same adjacencies as $x$ and 
$y$ respectively (we only construct those $R_{n, t}$ for which the above quantity is positive). 
We note that $|V(R_{n, t})| \leq n$ by construction, and $R_{n, t}$ is easily seen to be $3S_6$-free, so that $R_{n, t}$ is an induced subgraph of $H$. 
In addition, like with the original $x$ and $y$, each new pair $x \in X$ and $y \in Y$ has $|d_H(x, y) - t| \leq 2$.

For $3 \leq t \leq \left\lfloor\frac{n}{6}\right\rfloor - 2$, we have $|X| = |Y| \geq \left\lfloor\frac{n}{6}\right\rfloor$. 
In particular, each choice of $t \in I := \{3, 4, \dots, \left\lfloor\frac{n}{6}\right\rfloor - 2\} \cap \{3 + 5i : i \in \mathbb N\}$ 
witnesses the existence in $H$ of $|X||Y| \geq \left\lfloor\frac{n}{6}\right\rfloor^2$ pairs of vertices, and since 
the pairs' distance ranges for different $t \in I$ do not overlap, the sets of pairs are disjoint. 
Hence $H$ must contain in total at least 
$$|I|\left\lfloor\frac{n}{6}\right\rfloor^2 \geq  \frac{1}{5}\left(\left\lfloor\frac{n}{6}\right\rfloor - 5\right)\left\lfloor\frac{n}{6}\right\rfloor^2 = \Omega(n^3)$$ pairs, as claimed.
\end{proof}

Lemma \ref{lem:univlowerbound} shows indeed that there are other obstructions to a linear universal graph, 
but it is not yet clear what those obstructions are. For instance, the existence of a linear universal graph 
is a non-trivial question even for $S_t$-free graphs, i.e., bipartite permutation graphs of maximum degree at most $t - 1$. 
In fact, it is not clear whether every class with super-linear universal graphs contains a minimal such class, 
since the boundary class $\cal L$ described in Section~\ref{sec:wqo} has linear universal graphs. We leave the continuation of this study as an open problem.

\begin{problem}
Characterise the family of hereditary subclasses of bipartite permutation graphs that admit a universal bipartite permutation graph of linear order.
\end{problem}

We conclude the paper with one more related open problem. Theorem~\ref{thm:lettericity-bound} shows that the graph $H_{n,n}$ is not an optimal 
universal construction for the class of bipartite permutation graphs, because all graphs in this class can be embedded into 
$H_{n/2+1,n}$ as induced subgraphs. However, this construction is still quadratic. On the other hand, the following result provides an almost quadratic lower bound on the size of a universal graph.

\begin{theorem}\label{lem:univlowerbound2}
	Suppose $H$ is a bipartite permutation graph that contains all $n$-vertex bipartite permutation graphs as induced subgraphs. Then $|V(H)| = \Omega(n^\alpha)$ for any $\alpha < 2$.
\end{theorem}

\begin{proof}
We show $|V(H)| = \Omega(n^{2a-1/a})$ for each $a \in \mathbb N$. This is a generalisation of Lemma~\ref{lem:univlowerbound}, which deals with the case $a = 2$. 
	
The proof of Lemma~\ref{lem:univlowerbound} generalises as follows.
For $a \in \mathbb N$, we get $|V(H)| = \Omega(n^{2a-1/a})$ by counting $a$-sets of vertices. 
To do this, we associate to each $a$-set the $a \choose 2$-multiset consisting of distances between pairs of its vertices; 
we will refer to this $a \choose 2$-multiset as the ``distance multiset (in $H$)'' of the original $a$-tuple. 
In order to determine that two $a$-sets are distinct, it is enough to show they have distinct distance multisets.
	
We generalise the construction of the graphs $R_{n, t}$ to graphs $R_{n, T}$, where $T$ is a set of $a - 1$ natural numbers, 
each at least $3$. To construct $R_{n, T}$, we start with $Q_{\max(T)}$, but instead of inflating just the endpoints of the second path, 
we inflate the first vertex, then the $j + 3$rd, for each $j \in T$. 
By putting an appropriate upper bound (linear in $n$) on the size of elements in $T$, say $\lambda n$, 
we can arrange that each inflated set $X_j$ has size linear in $n$, while ensuring $|V(R_{n, T})| \leq n$. 
	
The set $T$ can be viewed as a condition on the distance multiset in $R_{n, T}$ of certain $a$-sets: 
an $a$-set consisting of one vertex from each inflated set $X_j$ has $\{t + 2: t \in T\}$ as a subset of its distance multiset in $R_{n, T}$. 
The distance multiset in $H$ might differ from the one in $R_{n, T}$, but like before, rigidity of the structure ensures the two are within a small tolerance of each other. 
Therefore, as long as we are careful in choosing what sets $T$ we consider, we can ensure that different choices of $T$ will give rise to different distance multiset subsets in $H$. 
This is achieved by choosing, like before, $T \subseteq \{3, \dots, \lambda n\} \cap \{3 + 5i : i \in \mathbb N\}$.
	
One last hurdle is the following: in order to decide that two $a$-sets of vertices are distinct, 
we actually need to compare them via their whole distance multisets, not just via the $a - 1$-subsets 
coming from the choice of $T$. We notice, however, that the same distance multiset can account (conservatively) for at most ${a \choose 2} \choose a  - 1$ different choices of $T$.
	
Altogether,	each choice of $T \subseteq \{3, \dots, \lambda n\} \cap \{3 + 5i : i \in \mathbb N\}$ 
witnesses the existence of $\Omega(n^a)$ $a$-sets of vertices in $H$, and each $a$-set is repeated overall 
at most a constant number of times. Since there are $\Omega(n^{a - 1})$ choices for $T$, 
this shows $|V(H)|^a \geq {|V(H)| \choose a} = \Omega(n^{2a - 1})$, from which $|V(H)| = \Omega(n^{2a - 1/a})$ as required.  
\end{proof}

We conjecture that the optimal universal graph is, in fact, quadratic. 

\begin{conjecture}
	The minimum number of vertices in a bipartite permutation graph containing all $n$ vertex bipartite permutation graphs is $\Omega(n^2)$. 
\end{conjecture}

Establishing the optimal constant would then be a problem analogous to the study of superpatterns from the world of permutations 
(see, for instance, \cite{superpattern321free, superpattern}\footnote{We remark that the study of superpatterns is usually done 
in the universe of {\it all} permutations, while in this paper we restrict ourselves to {\it bipartite} permutation graphs - 
this is the reason behind the apparent discrepancy between the upper bound from \cite{superpattern321free} and our lower bound from Lemma~\ref{lem:univlowerbound2}.}).


\end{document}